\documentclass[11pt]{amsart}
\usepackage[english]{babel}
\usepackage[utf8]{inputenc}
\usepackage{float}

\usepackage{amsmath,amssymb,amscd,amsthm,amsxtra,array}
\usepackage{color,graphicx}

\newtheorem{theorem}{Theorem}[section]
\newtheorem{lemma}[theorem]{Lemma}

\newtheorem{proposition}[theorem]{Proposition}

\RequirePackage{fix-cm}

\def\Xint#1{\mathchoice
{\XXint\displaystyle\textstyle{#1}}%
{\XXint\textstyle\scriptstyle{#1}}%
{\XXint\scriptstyle\scriptscriptstyle{#1}}%
{\XXint\scriptscriptstyle\scriptscriptstyle{#1}}%
\!\int}
\def\XXint#1#2#3{{\setbox0=\hbox{$#1{#2#3}{\int}$ }
\vcenter{\hbox{$#2#3$ }}\kern-.6\wd0}}

\def\dashint{\Xint-}

\numberwithin{equation}{section}

\begin{document}

\title[Weighted boundedness of $M^\otimes$]{Weighted boundedness of the 2-fold product of Hardy-Littlewood maximal operators}

\author[M. J. Carro]{Mar\'{\i}a J. Carro$^*$}

\address{Departament de Matem\`atiques i Inform\`atica, Universitat de Barcelona, 08007 Barcelona, Spain.} 
\email{carro@ub.edu}

\author[E. Roure]{Eduard Roure$^{**}$}
\address{Departament de Matem\`atiques i Inform\`atica, Universitat de Barcelona, 08007 Barcelona, Spain.} 
\email{eroure@ub.edu}

\thanks{The authors were supported by grants MTM2016-75196-P (MINECO / FEDER, UE) and 2014SGR289.\\ 
\indent\emph{E-mail addresses:} $^{*}$\texttt{carro@ub.edu}, $^{**}$\texttt{eroure@ub.edu}
}

\subjclass[2010]{42B99, 46E30}

\keywords{H\"older's inequality,  Muckenhoupt  weights}

\begin{abstract}
We study new weighted estimates  for the 2-fold product of Hardy-Littlewood maximal operators defined by $M^{\otimes}(f,g):= MfMg$. This operator appears very naturally in the theory of bilinear operators such as the bilinear Calder\'on-Zygmund operators, the bilinear Hardy-Littlewood maximal operator introduced by Calder\'on or in the study of  pseudodifferential operators.  To this end, we need to study H\"older's inequality for Lorentz spaces with change of measures 
$$
 \Vert fg \Vert_{L^{p,\infty}\left(w_1^{p/p_1} w_2^{p/p_2}\right)} \le 
C  \Vert f \Vert_{L^{p_1,\infty}(w_1 )}  \Vert g \Vert_{L^{p_2,\infty}(w_2 )}. 
$$
Unfortunately, we shall prove that  this inequality does not hold,  in general,  and we shall have to consider a weaker version of it. \end{abstract}

\maketitle

\section{Introduction and motivation}

Let us consider the Hardy-Littlewood maximal operator $M$, defined for locally integrable functions on $\mathbb R^n$ by
$$
Mf(x):=\sup_{Q\ni x}\frac{1}{|Q|}\int_Q |f(y)|dy,
$$
where the supremum is taken over all cubes $Q\subseteq\mathbb R^n$ containing $x$. The boundedness of $M$ in weighted Lebesgue spaces $L^p(w)$ is well understood since 1972 when Muckenhoupt \cite{Muckenhoupt TAMS72} proved that,  for every $1<p<\infty$, 
$$
M:L^p(w)\longrightarrow L^p(w) \quad\iff\quad w\in A_p,
$$
where $w\in A_p$ if $w$ is a positive and locally integrable function (called weight) such that 
$$
[ w]_{A_p}:= \sup_{Q} \left(\frac 1{|Q|} \int_Q w (x) dx\right) \left(\frac 1{|Q|} \int_Q w (x) ^{1-p'}dx \right)^{p-1} <\infty. 
$$

Moreover, in the context of weak type inequalities, if $1\le p<\infty$, 
$$
M:L^p(w)\longrightarrow L^{p, \infty}(w) \quad\iff\quad w\in A_p,
$$
where $w\in A_1$ if 
$$
Mw(x)\le C w(x), \quad \mbox{a.e. } x,
$$
and the infimum of all such constants $C$ in the above inequality is denoted by $[w]_{A_1}$.  Also, in the context of restricted weak type inequalities the following result was proved in \cite{chk:chk, KermanT SM82}: 
$$
M:L^{p,1}(w) \longrightarrow L^{p, \infty}(w) \quad\iff\quad w\in A_p^{\mathcal R},
$$
where a weight $w\in A_p^{\mathcal R}$ if 
\begin{equation*}
[w]_{A_p^{\mathcal R}} :=\sup_{Q} w(Q)^{1/p}\frac{\Vert \chi_Q w^{-1}\Vert_{L^{p',\infty}(w)}}{|Q|} <\infty.
\end{equation*}
For any measurable set $F$, we write $w(F)=\int_{F}w(x) dx$; if $w=1$, we simply write $|F|$. Moreover, 
\begin{equation*}
\Vert M\Vert_{L^{p,1}(w)\to L^{p,\infty}(w)} \lesssim [ w]_{A_p^{\mathcal R}}.
\end{equation*}

Our main goal in this paper is to study weighted estimates for the 2-fold product of Hardy-Littlewood maximal operators, defined for locally integrable functions $f$ and $g$ in the most simple way:
 $$
M^{\otimes}(f,g)(x):= Mf(x)Mg(x). 
$$

Most of the results in this paper can be extended to the $k$-fold product of Hardy-Littlewood maximal operators
  $$
M^{\otimes}(f_1, \dots, f_k)(x):= Mf_1(x)\cdots Mf_k(x), 
$$
but, for simplicity, we shall only present the case $k=2$.

This operator has been  very useful to obtain weighted estimates for several types of multilinear operators, such as the following ones: 
 
 \noindent
 1) The bilinear Hardy-Littlewood maximal operator: it was introduced by A. Calder\'on in 1964 and it is defined by 
 $$
 \mathcal  M(f,g)(x):=\sup_{r>0} \frac 1{|B(0, r)|}\int_{B(0,r)} |f(x-y)| |g(x+y)| dy. 
 $$
 
Using H\"older's inequality, we have that 
$$
 \mathcal M(f,g) \lesssim M(f^{1/\theta})^{\theta} M(g^{1/(1-\theta)})^{1-\theta}, 
 $$
 for every $0<\theta <1$ and hence, 
 \begin{equation}\label{lacey}
 \mathcal  M:L^{p_1}(\mathbb R^n)\times L^{p_2}(\mathbb R^n) \longrightarrow L^{p}(\mathbb R^n),
 \end{equation}
 for every $\frac 1p=\frac 1{p_1}+\frac 1{p_2}$ and  $p>1$,  and he conjectured that 
 $$
\mathcal  M:L^{2}(\mathbb R^n)\times L^{2}(\mathbb R^n) \longrightarrow L^{1}(\mathbb R^n).
 $$
 
 This conjecture was shown to be true by Lacey in 2000 (see \cite{l:l}),  proving the unexpected fact that \eqref{lacey} holds for $p_1>1$, $p_2>1$ and $p>2/3$.  In the same way, weighted estimates for the easier operator $M^{\otimes}$ will imply weighted estimates for $\mathcal M$. In particular, using H\"older's inequality,  one can immediately obtain that  
 \begin{equation}\label{mfold}
 M^{\otimes}: L^{p_1}(w_1)\times L^{p_2}(w_2) \longrightarrow L^p\left(w_1^{p/p_1} w_2^{p/p_2}\right), 
 \end{equation}
 for every $p_1, p_2> 1$, $\frac 1p=\frac 1{p_1}+\frac 1{p_2}$, $w_1\in A_{p_1}$ and 
 $w_2\in A_{p_2}$. Consequently, 
 $$
  \mathcal M: L^{p_1}(w_1)\times L^{p_2}(w_2) \longrightarrow L^p\left(w_1^{p/p_1} w_2^{p/p_2}\right), 
 $$
 for every $p_1 >1/\theta$, $p_2 >1/(1-\theta)$ and $w_1 \in A_{\theta p_1}$, $w_2 \in A_{(1-\theta)p_2}$. It is worth mentioning that much more delicate weighted estimates  for the  bilinear Hilbert transform have  been  recently obtained in \cite{diplinio}.

 \noindent
 2)  Let now $T$  be a bilinear Calder\'on-Zygmund operator;  that is, for every $f,g\in C^\infty_c$,  
 $$
 T(f,g)(x)=\int _{\mathbb R^{2n}} f(y_1) g(y_2) K(x, y_1, y_2) dy_1 dy_2, \quad \forall x\notin \mbox{supp }f\cap \mbox{supp }g, 
 $$
where $K$ is defined away from the diagonal $x=y_1=y_2$, satisfies the size estimate
 $$
 |K(y_0, y_1, y_2)| \lesssim \frac 1{(\sum_{k,l} |y_k-y_l|)^{2n}},
 $$
 and for some $\varepsilon>0$ it satisfies the regularity condition
 $$
 | K(y_0, y_1, y_2)-K(y_0, y_1', y_2)|\lesssim \frac{|y_1-y_1'|^\varepsilon}{(\sum_{k,l} |y_k-y_l|)^{2n+\varepsilon}},
 $$
 and similarly for $| K(y_0, y_1, y_2)-K(y_0, y_1,  y_2')|$. Let $T^*$ be its maximal truncated operator, defined by
 $$
 T^*(f,g)(x)=\sup_{\delta>0}\left|\int _{|x-y_1|^2+|x-y_2|^2>\delta^2} f(y_1) g(y_2) K(x, y_1, y_2) dy_1 dy_2\right|. 
 $$
 
Then, L. Grafakos and R. H. Torres proved in \cite{graftor}  the following bilinear Cotlar's inequality: for every $\eta>0$, there exists a finite constant $C_{\eta}$ such that for every $(f,g)\in L^{p_1}\times L^{p_2}$, with $1\leq p_1,p_2<\infty$, the following holds for all $x\in \mathbb R^n$:
 $$
 T^*(f,g)(x)\leq C_{\eta}\left( M(|T(f,g)|^\eta)(x)^{1/\eta}+M^{\otimes}(f,g)(x)\right).
 $$
 
 As a consequence, one can deduce that
 $$
 T^*: L^{p_1}(w_1)\times L^{p_2}(w_2) \longrightarrow L^{p}\left(w_1^{p/p_1}w_2^{p/p_2}\right),
 $$
with $w_1 \in A_{p_1}$ and $w_2 \in A_{p_2}$ by proving this estimate for the  easier operators $T$ and  $M^{\otimes}$.  In this setting of bilinear Carder\'on-Zygmund integral operators, many other results have been proved where the role of the operator $M^{\otimes}$ is fundamental (see, for example,  paper \cite{mn:mn} where a good-lambda estimate for a maximal CZ operator with kernel satisfying a Dini condition is proved).

Concerning weighted bounds for $M^{\otimes}$, the above easy exercise \eqref{mfold} becomes an open question when we want to characterize the weights for which 
\begin{equation}\label{openquestion}
 M^{\otimes}: L^{p_1,1}(w_1)\times L^{p_2, 1}(w_2) \longrightarrow L^{p,\infty}\left(w_1^{p/p_1} w_2^{p/p_2}\right), 
\end{equation}
and this is the question we want to address in this paper. In fact, the motivation comes from the recent restricted weak type Rubio de Francia extrapolation theory  (see \cite{cgs:cgs}, \cite{cs:cs}) where it has been proved that if an operator 
$$
T:L^{p,1}(w) \longrightarrow L^{p,\infty}(w), 
$$
for every $w\in A_p^{\mathcal R}$, then endpoint $(1,1)$ estimates  hold for characteristic functions; that is 
$$
\Vert T\chi_E\Vert_{L^{1,\infty}(u)} \lesssim u(E), \qquad\forall u\in A_1, 
$$
contrary to what happens with the classical Rubio de Francia theory. In this context, it has been proved in  \cite{cr:cr} that 
 the operator $ M^{\otimes}$  plays in the multilinear extrapolation theory of Rubio de Francia (see \cite{gm:gm})  the same role that the classical Hardy-Littlewood maximal operator plays  in the linear case of this theory. Therefore, the complete characterization of the weights $w_1$ and  $w_2$ satisfying \eqref{openquestion} becomes a fundamental and  interesting question.

 Obviously, if H\"older's inequality for Lorentz spaces with change of measures holds, as it happens with the $L^p$ spaces, 
\begin{equation}\label{holder}
 \Vert fg \Vert_{L^{p,\infty}\left(w_1^{p/p_1} w_2^{p/p_2}\right)} \le 
C  \Vert f \Vert_{L^{p_1,\infty}(w_1 )}  \Vert g \Vert_{L^{p_2,\infty}(w_2 )}, 
  \end{equation}
  and then, for every $w_j\in A_{p_j}^{\mathcal R}$, 
  $$
  \Vert M^{\otimes}(f,g)\Vert_{L^{p, \infty}(w_1^{p/p_1} w_2^{p/p_2})} \lesssim 
  \Vert Mf \Vert_{L^{p_1,\infty}(w_1 )}  \Vert Mg \Vert_{L^{p_2,\infty}(w_2 )}\lesssim \Vert f \Vert_{L^{p_1,1}(w_1 )}  \Vert g \Vert_{L^{p_2,1}(w_2 )},
  $$
as we expect. This is what happens in the particular case when all the weights are equal,  $w:=w_1=w_2=w_1^{p/p_1} w_2^{p/p_2}$, since 
\begin{eqnarray*}
\Vert fg\Vert_{L^{p,\infty}(w)}&=& \sup_{t>0 }t^{1/p} (fg)^*_{w}(t)   \lesssim  \sup_{t>0 }t^{1/p} (f)^*_{w}(t)  (g)^*_{w}(t)
\le  \Vert f   \Vert_{L^{p_1,\infty}(w)}\Vert g \Vert_{L^{p_2,\infty}(w)}.
\end{eqnarray*}
As a consequence:

\begin{theorem} If   $w \in A_{\min{\{p_1, p_2\}}}^{\mathcal R}$, then
$$
 M^{\otimes}: L^{p_1,1}(w)\times L^{p_2, 1}(w) \longrightarrow L^{p,\infty}(w).
$$
\end{theorem}

However, we shall prove in Section \ref{S-holder} that \eqref{holder} does not hold for general weights, and hence the complete solution to \eqref{openquestion} remains open.

On the other hand,  despite of the fact that \eqref{holder} is not true, the following more intriguing and difficult result was proved in \cite{loptt:loptt}: for every $w_1, w_2\in A_1$, 
 $$
 M^{\otimes}: L^{ 1}(w_1)\times L^{ 1}(w_2) \longrightarrow L^{1/2,\infty}\left(w_1^{1/2} w_2^{1/2}\right). 
$$
 
Let us just mention here that this is the endpoint of the general case \eqref{openquestion}, which is surprising because usually endpoint estimates are harder to prove than estimates where, for example,  all the spaces involved are Banach spaces. 
  
The paper is organized as follows: we shall devote Section \ref{S-holder} to give a counterexample of \eqref{openquestion} and prove a weaker version of H\"older's inequality for Lorentz spaces with change of measures which shall be fundamental for our main results in Section \ref{S-main} concerning weighted boundedness of $M^{\otimes}$.

Before going on, let us recall the definition of the spaces which are going to be important for us (see \cite{BennettS Book88}).  Given $p>0$ and an arbitrary measure $\nu$, $L^{p,1}(\nu)$ is the Lorentz space of measurable functions such that
$$
||f||_{L^{p,1}(\nu)}
:=p\int_0^\infty \lambda_f^\nu (y) ^{1/p}  dy  =\int_0^\infty  f_\nu^*(t) t^{1/p-1}  dt <\infty,
$$
and $L^{p,\infty}(\nu)$ is the Lorentz space of measurable functions such that
$$
||f||_{L^{p,\infty}(\nu)}:= \sup_{y>0}y  \lambda_f^\nu(y)^{1/p}= \sup_{t>0}t^{1/p}f_\nu^*(t)<\infty, 
$$
where  $f_\nu^*$ is the decreasing rearrangement of $f$ with respect to $\nu$, defined by 
$$
f_\nu^*(t):=\inf\{y>0:\lambda_f^\nu(y)\leq t\}, \qquad \lambda_f^\nu(t):=\nu(\{|f|>t\}).
$$

As usual, we write $A \lesssim B$ if there exists a positive constant $C>0$, independent of $A$ and $B$, such that $A\leq C B$. If $A\lesssim B$ and $B\lesssim A$, then we write $A\approx B$. 

Finally, unless indicated explicitly, we shall always assume that $1\le p_1,p_2<\infty$, although one can take $0<p_i<\infty$ whenever no conflict arises with the definition of the objects involved. Also, $w_1,w_2$ will denote weights on $\mathbb R^n$ and, by definition,  
$$
\frac 1p=\frac 1{p_1}+ \frac 1{p_2} \quad\mbox{and}\quad w=w_1^{p/p_1} w_2^{p/p_2}. 
$$

\section{H\"older's inequality for Lorentz spaces}\label{S-holder}

Let us start giving a counterexample that shows that \eqref{holder} does not hold for general weights. Let us consider $n\geq1$,  $0<p_1, p_2<\infty$ and 
$$
f(x)=\frac 1{|x|^{n/p_1}}\chi_{\{|x|\ge 1\}}, \ g(x)=   {|x|^{n/p_1}}\chi_{\{|x|\ge 1\}}, $$
 $$
 w_1(x)= 1, \ w_2(x)= \frac 1{|x|^{n \big(1+\frac{p_2}{p_1}\big)}} \chi_{\{|x|\ge 1\}}+\chi_{\{|x|<1\}}. 
$$
Then $fg= \chi_{\{|x|\ge 1\}}$ and $w(x)= w_2(x)^{p/p_2}= \frac 1{|x|^{n}} \chi_{\{|x|\ge 1\}}+\chi_{\{|x|<1\}}$,  and hence, 
$$
\Vert fg\Vert_{L^{p,\infty}(w)}^p=  
\int_{\{|x|\ge 1\}} \frac 1{|x|^{n}}   dx=+\infty,
$$
while $\Vert f\Vert_{L^{p_1, \infty}(w_1)}=\Vert f\Vert_{L^{p_1, \infty}}<\infty$,  and 
$$
\Vert g\Vert_{L^{p_2, \infty}(w_2)}\le  \sup_{s>0} s \left(\int_{\{x\in\mathbb R^n: |x|^{n/p_1}>s\}}\frac 1{|x|^{n \left(1+\frac{p_2}{p_1}\right)}}  dx\right)^{1/p_2}<\infty,
$$
and the result follows.

Due to this fact and in order to prove our main estimate for the operator $M^{\otimes}$,  we   need  the following weaker versions of H\"older's inequality.

\begin{lemma}\label{charac}
 Given a measurable set $E$ and a measurable function $g$, 
 $$
\left \| \chi_E g \right \|_{L^{p,\infty}(w)} \leq \left \| \chi_E \right \|_{L^{p_1,\infty}(w_1)} \left \| g \right \|_{L^{p_2,\infty}(w_2)}.
$$
 \end{lemma}
\begin{proof}
If the right-hand side is infinite, then there is nothing to prove, so we may assume that $w_1(E)<\infty$ and $\left \| g \right \|_{L^{p_2,\infty}(w_2)}<\infty$. 
Now, for every $t>0$,  we have that $\{\chi_E |g| >t\}=E\cap\{|g|>t\}$ and hence,  by H\"older's inequality,
\begin{align*}
t w(\{\chi_E |g| >t\})^{1/p} & \leq t w_1(\{\chi_E |g| >t\})^{1/p_1} w_2(\{\chi_E |g| >t\})^{1/p_2}
\\ & \leq t w_1(E)^{1/p_1} w_2(\{ |g| >t\})^{1/p_2},
\end{align*}
from which the result follows taking the supremum in $t>0$.
\end{proof}

\begin{lemma}\label{dholder2}
Given measurable functions $f$ and $g$, with $\left \| f \right \|_{\infty}\leq 1$, and $0<\delta <1$, we have that
$$
\left \| fg \right \|_{L^{p,\infty}(w)} \le C(p,\delta) \left \| f^{\delta} \right \|_{L^{p_1,\infty}(w_1)} \left \| g \right \|_{L^{p_2,\infty}(w_2)}.
$$
\end{lemma}
\begin{proof}

Let $F$ be a measurable function with $\left \| F \right \|_{\infty} \leq 1$. Then, by Lemma~\ref{charac}, we have that 
\begin{eqnarray*}
\sup_{0<t<1} t \left \| \chi_{\{|F|>t\}} g \right \|_{L^{p,\infty}(w)}& \leq &\sup_{0<t<1} t \left \| \chi_{\{|F|>t\}} \right \|_{L^{p_1,\infty}(w_1)} \left \| g \right \|_{L^{p_2,\infty}(w_2)}
\\
&\le&  \left \| F \right \|_{L^{p_1,\infty}(w_1)} \left \| g \right \|_{L^{p_2,\infty}(w_2)}.
\end{eqnarray*}

 Let $0<\delta <1$, set $F=f^{\delta}$ and fix  $0<q<p$. By Kolmogorov's inequality (see \cite{gz:gz} or   \cite[Ex.~1.1.12]{grafclas}),
$$
\left \| f g \right \|_{L^{p,\infty}(w)} \leq \sup_{0<w(A)<\infty} \left \| fg\chi_A \right \|_{L^{q}(w)}w(A)^{1/p-1/q},
$$ 
where the supremum is taken over all measurable sets $A$ with $0<w(A)<\infty$. For one of such sets $A$, we have that
\begin{align*}
\left \| fg\chi_A \right \|_{L^{q}(w)}^q & =\sum_{k<0} \int_{A \cap \{2^k < |f| \leq 2^{k+1}\}} |fg|^q w \leq 2^q \sum_{k<0} 2^{kq} \left \| \chi_{\{|f|>2^k\}} g\chi_A \right \|_{L^{q}(w)} ^q
\\ &  \leq 2^q \sum_{k<0} 2^{k(1-\delta)q}    \bigg( 2^{\delta k}  \left \| \chi_{\{|f|^\delta >2^{\delta k}\}} g\chi_A \right \|_{L^{q}(w)} \bigg)^q
\\ & \leq \frac{2^q}{2^{(1-\delta)q}-1} \left( \sup_{0<t<1} t \left \| \chi_{\{|f|^{\delta}>t\}} g\chi_A \right \|_{L^{q}(w)} \right)^q,
\end{align*}  
and hence,  applying Kolmogorov's inequality again,
\begin{align*}
& \sup_{0<w(A)<\infty} \left \| fg\chi_A \right \|_{L^{q}(w)}w(A)^{1/p-1/q} \\ &
\leq 2 (2^{(1-\delta)q}-1)^{-1/q} \sup_{0<t<1} t \sup_{0<w(A)<\infty} \left \| \chi_{\{|f|^{\delta}>t\}} g\chi_A \right \|_{L^{q}(w)} w(A)^{1/p-1/q} \\ & \leq 2 (2^{(1-\delta)q}-1)^{-1/q} \left(\frac{p}{p-q}\right)^{1/q}\sup_{0<t<1} t  \left \| \chi_{\{|f|^{\delta}>t\}} g \right \|_{L^{p,\infty}(w)} \\ &  \leq 2 (2^{(1-\delta)q}-1)^{-1/q} \left(\frac{p}{p-q}\right)^{1/q} \left \| f^{\delta} \right \|_{L^{p_1,\infty}(w_1)} \left \| g \right \|_{L^{p_2,\infty}(w_2)}
\\
&\leq 2 \left(\frac{p}{\log 2(1-\delta)q(p-q)} \right)^{1/q} \left \| f^{\delta} \right \|_{L^{p_1,\infty}(w_1)} \left \| g \right \|_{L^{p_2,\infty}(w_2)}.
\end{align*}

Hence, the theorem follows taking 
\begin{align*}
C(p,\delta)&:= \inf_{0<q<p} 2 \left(\frac{p}{\log 2(1-\delta)q(p-q)} \right)^{1/q} = 2 \left(\inf_{0<\theta<1}  \left(\log 2(1-\delta)p\theta(1-\theta) \right)^{-1/\theta}\right)^{1/p}.
\end{align*}
\end{proof}

Observe that if $p>1$, then $C(p,\delta) \lesssim_p \frac{1}{1-\delta}$, and if $p\leq 1$, then $C(p,\delta) \lesssim_{p,\alpha} \frac{1}{(1-\delta)^{1+\alpha}}$, for every $\alpha>-1/p'$.

\section{Main results}\label{S-main}

For our first lemma,  recall  that the class $A_\infty$ is simply defined by $A_\infty:=\bigcup_{p\ge 1}A_p$. See also the recent paper \cite{dk:dk} for related results. 

\begin{lemma}\label{cubes}
If $w_1,w_2\in A_{\infty}$, then for every cube $Q$,
\begin{equation}\label{cubosr}
w_1(Q)^{1/p_1}w_2(Q)^{1/p_2} \approx w(Q)^{1/p}.
\end{equation}
\end{lemma}

\begin{proof}
Since $w_1,w_2\in A_{\infty}$,   we have by Theorem 2.1 in \cite{cn:cn} that  $w_i ^{p/p_i} \in RH_{p_i/p}$, $i=1,2$, and the result follows by Theorem 2.6 in the same paper. 
\end{proof}

\begin{proposition}\label{apr}
If $w_1, w_2\in A_{\infty}$,  the following statements are equivalent:
\begin{enumerate}
\item $w_2 \in A_{p_2}^{\mathcal R}$.
\item For every measurable set $E$ and every measurable function $g$,
$$
\left \| \chi_E Mg \right \|_{L^{p,\infty}(w)} \lesssim w_1(E)^{1/p_1} \left \| g \right \|_{L^{p_2,1}(w_2)}. 
$$
\item For every cube $Q$ and every measurable function $g$,
$$
\left \| \chi_Q Mg \right \|_{L^{p,\infty}(w)} \lesssim w_1(Q)^{1/p_1} \left \| g \right \|_{L^{p_2,1}(w_2)}. 
$$
\end{enumerate}
\end{proposition}

\begin{proof}
2) follows from 1) applying Lemma~\ref{charac}, and it is clear that 3) follows from 2). Let us show that 3) implies 1). Fix a cube $Q$ and, using duality, choose a non-negative $g$ such that $\left \| g \right \|_{L^{p_2,1}(w_2)} \leq 1$ and
$$
 \left \| \chi_Q w_2 ^{-1} \right \|_{L^{p'_2,\infty}(w_2)} \lesssim \int g (\chi_Q w_2^{-1}) w_2=\int _Q g.
$$

Since $Mg\geq \left( \dashint_Q g\right)\chi_Q$, 3) implies that
$$
  \frac{w(Q)^{1/p}}{|Q|}\left \| \chi_Q w_2 ^{-1} \right \|_{L^{p'_2,\infty}(w_2)} \lesssim \left( \dashint _Q g\right) w(Q)^{1/p} \lesssim w_1(Q)^{1/p_1}.
$$

Applying Lemma~\ref{cubes}, we get that
$$
\frac{w_2(Q)^{1/p_2}}{|Q|}\left \| \chi_Q w_2 ^{-1} \right \|_{L^{p'_2,\infty}(w_2)}\lesssim 1,
$$
and taking the supremum over all cubes $Q$, we get that $w_2\in A_{p_2}^{\mathcal R}$.
\end{proof}

As a consequence we obtain our first main result. 

\begin{theorem}
If  $w_1,w_2\in A_{\infty}$ and \eqref{openquestion} holds, then 
$w_i\in A_{p_i}^{\mathcal R}$, $i=1,2$.
\end{theorem}

\begin{proof} Since 
 $\chi_E\le M\chi_E  $, we have that 
$$
\left \| \chi_E Mf_2 \right \|_{L^{p,\infty}(w)}\le \left \| M\chi_E Mf_2 \right \|_{L^{p,\infty}(w)}= \left \| M^{\otimes}(\chi_E, f_2) \right \|_{L^{p,\infty}(w)}  \lesssim w_1(E)^{1/p_1} \left \| f_2 \right \|_{L^{p_2,1}(w_2)},
$$
and similarly,
$$
\left \| \chi_E Mf_1 \right \|_{L^{p,\infty}(w)} \le \left \|  Mf_1M\chi_E\right \|_{L^{p,\infty}(w)}= \left \| M^{\otimes}(f_1, \chi_E) \right \|_{L^{p,\infty}(w)}   \lesssim w_2(E)^{1/p_2} \left \| f_1 \right \|_{L^{p_1,1}(w_1)}.  
$$
The desired result follows from Proposition~\ref{apr}.
\end{proof}

Observe that, in fact, the hypotheses that $w_1,w_2\in A_{\infty}$   can be replaced by \eqref{cubosr}. 

We believe that the converse of the previous result holds. However, up to now, we need to assume some stronger condition in one of the weights;  namely, either  $w_1\in A_{p_1}$ or  $w_2\in A_{p_2}$.

\begin{theorem}
Let $p_1 >1$ and let $w_1 \in A_{p_1}$ and $w_2 \in A_{p_2}^{\mathcal R}$. Then, for every measurable set $E$ and every measurable function $g$, 
$$
\left \| M^{\otimes}(\chi_E, g)   \right \|_{L^{p,\infty}(w)} \lesssim \left \| \chi_E \right \|_{L^{p_1,1}(w_1)}\left \| g \right \|_{L^{p_2,1}(w_2)}.
$$
\end{theorem}

\begin{proof}  Since $w_1 \in A_{p_1}$ and $p_1>1$, there exists $0<\delta < 1$ such that $w_1 \in A_{p_1 \delta}$. Applying Lemma~\ref{dholder2}, we obtain that
\begin{align*}
\left \| M^{\otimes}(\chi_E, g) \right \|_{L^{p,\infty}(w)} & \lesssim \left \| M\chi_E  \right \|_{L^{p_1 \delta,\infty}(w_1)} ^{\delta} \left \|  Mg \right \|_{L^{p_2,\infty}(w_2)} \\ & \lesssim \left \| \chi_E \right \|_{L^{p_1 \delta, 1}(w_1)}^{\delta} \left \| g \right \|_{L^{p_2,1}(w_2)} \approx \left \| \chi_E \right \|_{L^{p_1,1}(w_1)} \left \| g \right \|_{L^{p_2,1}(w_2)}.
\end{align*}
\end{proof}

Observe that if $p>1$,  this result can be extended to arbitrary measurable functions $f$ and $g$ by using that $L^{p,\infty}(w)$ is a Banach space.

\begin{theorem} If  $p, p_1>1$,  $w_1 \in A_{p_1}$ and $w_2 \in A_{p_2}^{\mathcal R}$ (or 
  $p, p_2>1$,  $w_1 \in A_{p_1}^{\mathcal R}$ and $w_2 \in A_{p_2}$), 
then 
$$
 M^{\otimes}: L^{p_1,1}(w_1)\times L^{p_2, 1}(w_2) \longrightarrow L^{p,\infty}\left(w_1^{p/p_1} w_2^{p/p_2}\right). 
$$
\end{theorem}

Similar results can be proved for  $A_p$ weights. We state them without proofs since these are completely analogous. 

\begin{proposition}
If $w_1,w_2\in A_{\infty}$, then the  following statements are equivalent:
\begin{enumerate}
\item $w_2 \in A_{p_2}$.
\item For every measurable set $E$ and every measurable function $g$,
$$
\left \| \chi_E Mg \right \|_{L^{p,\infty}(w)} \lesssim w_1(E)^{1/p_1} \left \| g \right \|_{L^{p_2}(w_2)}. 
$$
\item For every cube $Q$ and every measurable function $g$,
$$
\left \| \chi_Q Mg \right \|_{L^{p,\infty}(w)} \lesssim w_1(Q)^{1/p_1} \left \| g \right \|_{L^{p_2}(w_2)}. 
$$
\end{enumerate}
\end{proposition}

\begin{proposition}
If $w_1,w_2\in A_{\infty}$  and $p_2>1$, then the  following statements are equivalent:
\begin{enumerate}
\item $w_2 \in A_{p_2}$.
\item For every measurable set $E$ and every measurable function $g$,
$$
\left \| \chi_E Mg \right \|_{L^{p}(w)} \lesssim w_1(E)^{1/p_1} \left \| g \right \|_{L^{p_2}(w_2)}. 
$$
\item For every cube $Q$ and every measurable function $g$,
$$
\left \| \chi_Q Mg \right \|_{L^{p}(w)} \lesssim w_1(Q)^{1/p_1} \left \| g \right \|_{L^{p_2}(w_2)}. 
$$
\end{enumerate}
\end{proposition}

\begin{theorem}
If $w_1,w_2\in A_{\infty}$, then 
$$
 M^{\otimes}: L^{p_1}(w_1)\times L^{p_2}(w_2) \longrightarrow L^{p, \infty}\left(w_1^{p/p_1} w_2^{p/p_2}\right), 
$$
 if and only if $w_i\in A_{p_i}$, $i=1,2$. And if $1<p_1, p_2<\infty$, this last condition is also equivalent to 
$$
 M^{\otimes}: L^{p_1}(w_1)\times L^{p_2}(w_2) \longrightarrow L^{p}\left(w_1^{p/p_1} w_2^{p/p_2}\right).  
$$
\end{theorem}

\end{document}